\documentclass[12pt]{amsart} 
\usepackage[utf8]{inputenc}
\makeatletter
\@namedef{subjclassname@2010}{%
  \textup{2010} Mathematics Subject Classification}
\makeatother
\usepackage{amssymb,color,verbatim}
\newtheorem{theorem}{Theorem}
\newtheorem{cor}[theorem]{Corollary}

\newtheorem{lemma}[theorem]{Lemma}

\newtheorem{remark}[theorem]{Remark}

\newcommand{\bs}{\bigskip}
\newcommand{\n}{\noindent}
\newcommand{\ds}{\displaystyle}

\newcommand{\Ba}{\ensuremath{\mathcal{B}}}
\newcommand{\Pa}{\ensuremath{\mathcal{P}}}

\newcommand{\F}{\ensuremath{\mathbb{F}}}
\newcommand{\ta}{\Theta}
\newcommand{\de}{\delta}

\newcommand{\pg}{{\rm PG}}
\newcommand{\md}{\;{\rm mod}\;}
\newcommand{\Hp}{{\mathcal H}}
\begin{document}
\title[The maximum size of a partial spread]{The maximum  size of a partial spread in a finite projective space}
\author[Esmeralda L. N\u{a}stase and Papa A. Sissokho]{\tiny{Esmeralda L. N\u{a}stase} \\ \\
 Mathematics Department\\ Xavier University\\
Cincinnati, Ohio 45207, USA\\ \\
Papa A. Sissokho  \\ \\
Mathematics Department \\ Illinois State University\\ 
Normal, Illinois 61790, USA}
\thanks{nastasee@xavier.edu,  psissok@ilstu.edu}
\begin{abstract}
Let $n$ and $t$ be positive integers with $t<n$, and let $q$ be a prime power. 
A {\em partial $(t-1)$-spread} of ${\rm PG}(n-1,q)$ is a set of $(t-1)$-dimensional subspaces of ${\rm PG}(n-1,q)$ 
that are pairwise disjoint.
Let $r=n\mbox{ mod } t$ and $0\leq r<t$. We prove that  if $t>(q^r-1)/(q-1)$,
 then the maximum size, i.e., cardinality, of a partial $(t-1)$-spread of ${\rm PG}(n-1,q)$ is $(q^n-q^{t+r})/(q^t-1)+1$. 
 This essentially settles a main open problem in this area. Prior to this 
 result, this maximum size was only known for $r\in\{0,1\}$ and for $r=q=2$.
\end{abstract}
\maketitle
\n  

\bs\n
\section{Introduction}\label{sec:1}
Let $n$ and $t$ be positive integers with $t<n$, and let $q$ be a prime power.
Let $\pg(n-1,q)$ denote the $(n-1)$-dimensional 
projective space over the finite field $\F_q$.  
A {\em partial $(t-1)$-spread} $S$ of $\pg(n-1,q)$ is a collection 
of $(t-1)$-dimensional subspaces of 
$\pg(n-1,q)$ that are pairwise disjoint.
If $S$ contains all the points of $\pg(n-1,q)$, then it is called 
a {\em $(t-1)$-spread}. It follows from the work of Andr\'e~\cite{An} (also see~\cite[p. 29]{De}) that a $(t-1)$-spread of $\pg(n-1,q)$ exists if and only if $t$ divides $n$.

Given positive integers $n$ and $t$ with $t<n$, the problem of finding the maximum size, i.e., cardinality, of a partial $(t-1)$-spread of 
$\pg(n-1,q)$ is rather a natural one.
It is directly related to the general problem of classifying the {\em maximal} partial $(t-1)$-spread.
A maximal partial $(t-1)$-spread is a set of pairwise disjoint $(t-1)$-dimensional subspaces which cannot be extended to a 
larger set. This problem has been extensively studied~\cite{EiStSz,GaSz,He1,JuSt}. 
Besides their traditional relevance to Galois geometry, partial $(t-1)$-spreads are used to build byte-correcting codes (e.g., see~\cite{Et1,HP}), $1$-perfect mixed
error-correcting codes (e.g., see~\cite{HeSc,HP}), orthogonal arrays and $(s,k,\lambda)$-nets (e.g., see~\cite{DF}). More recently, partial $(t-1)$-spreads have also attracted renewed attention since they can be viewed as {\em subspace codes}. In Section~\ref{sec:conc}, we shall say more about the connection between our results and   subspace codes.

Let $\mu_q(n,t)$ denote the maximum size
of any partial $(t-1)$-spread of $\pg(n-1,q)$.
The problem of determining $\mu_q(n,t)$ is a long standing open problem. 
A general upper bound for $\mu_q(n,t)$ is given by the following theorem of Drake and Freeman~\cite{DF}.
\begin{theorem}\label{DF}
Let $r=n\md{t}$ and $0\leq r<t$. Then 
\[\mu_q(n,t)\leq \ds\frac{q^n-q^r}{q^t-1}-\lfloor\omega\rfloor-1,\]
where $2\omega=\sqrt{4q^t(q^t-q^r)+1}-(2q^t-2q^r+1)$.
\end{theorem}
The following result is due to Andr\'e~\cite{An} for $r=0$. For $r=1$, it is due to
Hong and Patel~\cite{HP} when $q=2$, 
and Beutelspacher~\cite{Be} when $q>2$.
\begin{theorem}\label{BHP}
Let $r=n\md{t}$ and $0\leq r<t$. Then 
\[\mu_q(n,t)\geq \ds\frac{q^n-q^{t+r}}{q^t-1}+1,\] 
where equality holds if $r\in\{0,1\}$.
\end{theorem}
In light of Theorem~\ref{BHP}, it was conjectured (e.g., see~\cite{EiSt,HP}) that the value of $\mu_q(n,t)$ is given by the lower bound in Theorem~\ref{BHP}. However, this conjecture was disproved by El-Zanati et al.~\cite{EJSSS} who proved the following result.
\begin{theorem}\label{EJSSS} If $n\geq8$ and $n\md{3}=2$, then
$\mu_2(n,3)= \ds\frac{2^n-2^{5}}{7}+2$.
\end{theorem}
Very recently, Kurz~\cite{K} proved the following theorem which upholds the lower bound for $\mu_q(n,t)$ when $q=2$, $r =2$, and 
$t > 3$. 
\begin{theorem}\label{K} If $n>t>3$ and $n\md{t}=2$, then
\[\mu_2(n,t)=\ds\frac{2^n-2^{t+2}}{2^t-1}+1.\]
\end{theorem}
In this paper, we prove that the conjectured value of $\mu_q(n,t)$ holds for almost all values of the parameters $n$, $q$, and $t$. 
The following theorem, which is our main result, 
 generalizes Theorem~\ref{BHP} (set $r=0$ or $r=1$) and Theorem~\ref{K} (set $r=2$ and $q=2$). In particular, this is the first comprehensive result with the exact value of $\mu_q(n,t)$ for almost all values of the parameters $n$, $q$, and $t$. 
\begin{theorem}\label{thm1}
Let $r=n\md{t}$ and $0\leq r<t$. If $t>(q^r-1)/(q-1)$, then 
\[\mu_q(n,t)=\ds\frac{q^n-q^{t+r}}{q^t-1}+1.\]
\end{theorem}

We can use the language of graph theory to reformulate Theorem~\ref{thm1} as follows. 
Let $\Hp_q(n,t)$ be 
the hypergraph whose vertices are the points of $\pg(n-1,q)$
and whose edges are its $(t-1)$-subspaces. Then  $\Hp_q(n,t)$ 
is a $(q^t-1)/(q-1)$-uniform hypergraph.  Now Theorem~\ref{thm1}
implies that if $t>(q^r-1)/(q-1)$, then the maximum size of a matching in $\Hp_q(n,t)$ is $(q^n-q^{t+r})/(q^t-1)+1$.

\bs The general strategy of the proof of Theorem~\ref{thm1}
is due to Beutelspacher who used it to 
prove Theorem~\ref{BHP}. This strategy relies on {\em subspace partitions} which we shall discuss in Section~\ref{sec:2}.
Beutelspacher's  approach was extended 
by Kurz to prove Theorem~\ref{K}. In this paper, we developed 
an averaging argument, which allows us to fully extend Beutelspacher's method and prove our main result (see Theorem~\ref{thm1}) in Section~\ref{sec:3}.
\section{Subspace partitions}\label{sec:2}
Let $V=V(n,q)$ denote the vector space of dimension $n$ over $\F_q$.
For any subspace $U$ of $V$, let $U^*$ denote the set of nonzero vectors in $U$.  A  {\em $d$-subspace} of $V(n,q)$
is a $d$-dimensional subspace of $V(n,q)$; this is equivalent to a {\em $(d-1)$-subspace} in $\pg(n-1,q)$.

A {\em subspace partition} $\Pa$ of $V$, also known as a {\em vector space partition}, is a collection of 
nontrivial subspaces of $V$ such that each vector of $V^*$ is in exactly one subspace of $\Pa$ (e.g., see Heden~\cite{He1} for 
a survey on subspace partitions). 
The {\em size} of a subspace partition $\Pa$ is the number of subspaces in $\Pa$. 

Suppose that there are $s$ distinct vector space dimensions, $d_s>\dots>d_1$, that occur as dimensions of subspaces in a subspace 
partition $\Pa$, and let $n_i$ denote the number of $i$-subspaces in $\Pa$. Then the expression $[d_s^{n_{d_s}},\ldots,d_1^{n_{d_1}}]$ is called the {\em type} of $\Pa$.  
\begin{remark}\label{rmk:eq}
A partial $(t-1)$-spread of $\pg(n-1,q)$ of size $n_t$ is a partial 
$t$-spread of $V(n,q)$ of size $n_t$. This is 
equivalent to a subspace partition of $V(n,q)$ of type 
$[t^{n_t},1^{n_1}]$. 
We will use this subspace partition formulation in the proof 
of Lemma~\ref{lem:main}.
\end{remark}

\bs To state the next lemmas, we need the following definitions.
For any integer $i\geq 1$, let 
\[\ta_i=\frac{q^i-1}{q-1}.\] 
Then, for $i\geq 1$, $\ta_i$ is the number of $1$-subspaces in an $i$-subspace of $V(n,q)$. 
Let $\Pa$ be a subspace partition of $V=V(n,q)$ of type 
$[d_s^{n_{d_s}},\ldots,d_1^{n_{d_1}}]$.
For any hyperplane $H$ of $V$, let $b_{H,d}$ be the number of $d$-subspaces in $\Pa$ that are contained in $H$ and set 
$b_{H}=[b_{H,d_s},\ldots,b_{H,d_1}]$.  
Define the set $\Ba$ of {\em hyperplane types} as follows:
\[\Ba=\{b_H:\; \mbox{$H$ is a hyperplane of $V$}\}.\]
For any $b \in \Ba$, let $s_{b}$ denote the number of hyperplanes of $V$ of type $b$.

\bs
We will also use Lemma~\ref{lem:HeLe0} and Lemma~\ref{lem:HeLe1} by Heden and Lehmann~\cite{HeLe}. 
\begin{lemma}\label{lem:HeLe0}
Let $\Pa$ be a subspace partition of $V(n,q)$ of type 
$[d_s^{n_{d_s}},\ldots,d_1^{n_{d_1}}]$. If $H$ is a hyperplane 
of $V(n,q)$ and $b_{H,d}$ is as defined above, then
\[|\Pa|=1+\sum\limits_{i=1}^s b_{H,d_i}q^{d_i}.\]
\end{lemma}
\begin{lemma}\label{lem:HeLe1}
Let $\Pa$ be a subspace partition of $V(n,q)$, and let $\Ba$ and $s_b$ be as defined above. Then 
\[\sum\limits_{b\in \Ba} s_b=\ta_n,
\]
and for any $d$-subspace of $\Pa$, the following holds:
\[\sum\limits_{b\in \Ba}b_ds_b=n_d \ta_{n-d}.\]
\end{lemma}
\section{Proof of Theorem~\ref{thm1}}\label{sec:3}
We use the following notation throughout this section. Let
\begin{equation}\label{l}
\ell=\frac{q^{n-t}-q^{r}}{q^t-1}.
\end{equation}
Then the lower bound for $\mu_q(n,t)$ in Theorem~\ref{BHP} can be written as:
\[\mu_q(n,t)\geq \ell q^t+1.\]
\bs We now prove our main lemma. 
\begin{lemma}\label{lem:main} 
 Let $q$ be a prime power. Let $n$, $t$, and $r$ be integers such that $0\leq r<t<n$ and $r=n\md{t}$.
If $r\geq 1$ and $t>\ta_r$, then 
\[\mu_q(n,t)\leq \ell q^t+1.\]
\end{lemma}
\begin{proof} 
Recall that $\ta_i=(q^i-1)/(q-1)$ for any integer $i\geq 1$. 
For convenience, we also set 
\[\de_i=\frac{q^{i}-2q^{i-1}+1}{q-1}.\]
Since $q\geq2$, we have the following easy facts, which we will use throughout the proof.
\begin{equation}\label{eq:dlt}
0<\de_i<q^{i-1};\;\de_i\md{q^{i-1}}=\de_i;\;1+\de_{i+1}=q\de_{i};\mbox{ and } \frac{\de_{i+1}}{q}<\de_{i}.
\end{equation}

The proof is by contradiction. So assume that $\mu_q(n,t)>\ell q^t+1$. Then $\pg(n-1,q)$ has a
 $(t-1)$-partial spread of size $\ell q^t+2$. Thus, it follows from Remark~\ref{rmk:eq} that there exists a subspace partition $\Pa_0$ of $V(n,q)$ of type $[t^{n_t},1^{n_1}]$,
where 
\begin{multline}\label{eq:ntn1}
n_t=\ell q^t+2\mbox{ and } \\ 
n_1=\left(\frac{q^r-1}{q-1}-1\right)q^t +\frac{q^{t+1}-2q^t+1}{q-1}=(\ta_r-1)q^{t} +\de_{t+1}.
\end{multline}
We will prove by induction that for each integer $j$ with 
$0\leq j\leq \ta_r-1$, there exists
a subspace partition $\Pa_{j}$ of $H_{j}\cong V(n-j,q)$ of type 
\begin{equation}\label{ih}
[t^{m_{j,t}},(t-1)^{m_{j,t-1}}, \dots,(t-j)^{m_{j,t-j}},1^{m_{j,1}}],
\end{equation}
where $m_{j,t},\ldots,m_{j,t-j}$, $m_{j,1}$, and $c_j$ are nonnegative integers
such that 
\begin{equation}\label{prp1}
\sum_{i=t-j}^t m_{j,i}=n_t=\ell q^t+2,
\end{equation} 
\begin{equation}\label{prp2}
m_{j,1}=c_jq^{t-j}+\de_{t+1-j},\mbox{ and } 0\leq c_j\leq \ta_r-1-j.
\end{equation} 

The base case, $j=0$, holds since $\Pa_0$ is 
a subspace partition of $H_0=V(n,q)$ with type $[t^{n_t},1^{n_1}]$, and with 
the properties given in~\eqref{eq:ntn1}, which thus satisfies the conditions specified in \eqref{ih},~\eqref{prp1}, and~\eqref{prp2}.

For the inductive step, suppose that for some $j$, with $0\leq j<\ta_r-1$, we have constructed a subspace partition $\Pa_{j}$ of $H_{j}\cong V(n-j,q)$ of the type given in~\eqref{ih}, and with the properties given in~\eqref{prp1} and~\eqref{prp2}.
We then use Lemma~\ref{lem:HeLe1} to determine the average, $b_{avg,1}$, of the values $b_{H,1}$ over all hyperplanes $H$ of $H_{j}$.
\begin{align}\label{eq:5}
b_{avg,1}
=\frac{m_{j,1} \ta_{n-1-j}}{ \ta_{n-j}}
&=\left(c_jq^{t-j}+\de_{t+1-j}\right)\left(\frac{q^{n-1-j}-1}{q^{n-j}-1}\right)\cr
&<\frac{c_jq^{t-j}+\de_{t+1-j}}{q}\\
&<c_jq^{t-j-1}+\de_{t-j}.\notag
\end{align} 
It follows from~\eqref{eq:5} that there exists a hyperplane $H_{j+1}$ of $H_{j}$ with 
\begin{equation}\label{eq:6}
b_{H_{j+1},1} \leq b_{avg,1} < c_jq^{t-j-1}+\de_{t-j}.
\end{equation}
Next, we apply Lemma~\ref{lem:HeLe0} and~\eqref{eq:dlt}
to the partition 
$\Pa_j$ and the hyperplane $H_{j+1}$ of $H_{j}$ to obtain:
\begin{align}\label{eq:7.1}
1+b_{H_{j+1},1}\;q+\sum_{i=t-j}^t b_{H_{j+1},i}\;q^i=|\Pa_j|
&=n_t+m_{j,1}\cr
&=\ell q^{t}+2+c_jq^{t-j}+\de_{t+1-j}\cr
&=1+\ell q^{t}+c_jq^{t-j}+q\de_{t-j},
\end{align}
where $0\leq c_j\leq \ta_r-1-j$. Simplifying~\eqref{eq:7.1} yields
\begin{equation}\label{eq:7.2}
b_{H_{j+1},1}+\sum_{i={t-j}}^{t} b_{H_{j+1},i}\;q^{i-1}=\ell q^{t-1}+c_jq^{t-j-1}+\de_{t-j}.
\end{equation}
Then, it follows from~\eqref{eq:dlt} and~\eqref{eq:7.2} that
\begin{equation}\label{eq:8}
b_{H_{j+1},1}\md{q^{t-j-1}}=\de_{t-j}.
\end{equation}
By~\eqref{eq:6} and~\eqref{eq:8}, there exists a nonnegative integer $c_{j+1}$ such that 
\begin{equation}\label{prp2+}
m_{j+1,1}=b_{H_{j+1},1}=c_{j+1}q^{t-j-1}+\de_{t-j},
\mbox{ and } 0\leq c_{j+1}\leq \ta_r-2-j.
\end{equation}
Let $\Pa_{j+1}$ be the subspace partition of 
$H_{j+1}$ defined by:
\[\Pa_{j+1}=\{W\cap H_{j+1}:\; W\in \Pa_{j}\}.\]
Since $t-j>2$ (because $j+1<\ta_r<t$) and $\dim(W\cap H_{j+1})\in\{\dim W,\dim W-1\}$ for 
each $W\in\Pa_{j}$, it follows that $\Pa_{j+1}$ is a subspace partition of $H_{j+1}$ of type 
\begin{equation}\label{ih+}
[t^{m_{j+1,t}},(t-1)^{m_{j+1,t-1}}, \dots,(t-j-1)^{m_{j+1,t-j-1}},1^{m_{j+1,1}}],
\end{equation}
where $m_{j+1,t},m_{j+1,t-1},\ldots,m_{j+1,t-j-1}$ satisfy
\begin{equation}\label{prp1+}
\sum_{i=t-j-1}^t m_{j+1,i}=\sum_{i=t-j}^t m_{j,i}=n_t.
\end{equation}

The inductive step follows since $\Pa_{j+1}$ is a subspace partition of $H_{j+1}\cong V(n-j-1,q)$
of the type given in~\eqref{ih+}, which satisfies the conditions in~\eqref{prp1+} and~\eqref{prp2+}. 

Thus far, we have shown that the desired subspace partition $\Pa_j$ of $H_j$ exists
for any integer $j$ such that  $0\leq j\leq \ta_r-1$. 

\bs For the final part of the proof, we set $j=\ta_r-1$ 
and show that the existence of the subspace partition $\Pa_{\ta_r-1}$
of $H_{\ta_r-1}$ leads to a contradiction. If $j=\ta_r-1$, then it follows from~\eqref{prp2} that $c_{\ta_r-1}=0$ and $m_{\ta_r-1,1}=\de_{t+2-\ta_r}$.  
We use Lemma~\ref{lem:HeLe1} one last time to determine the average, $b_{avg,1}$, of the values $b_{H,1}$ over all hyperplanes $H$ of $H_{\ta_r-1}$. 
We obtain, 
\begin{align}\label{eq:9}
b_{avg,1}
=\frac{m_{\ta_r-1,1} \ta_{n-\ta_r}}{\ta_{n-\ta_r+1}}
&=\de_{t+2-\ta_r}\frac{q^{n-\ta_r}-1}{q^{n-\ta_r+1}-1}\cr
&<\frac{\de_{t+2-\ta_r}}{q}\cr
&<\de_{t+1-\ta_r}.
\end{align} 
It follows from~\eqref{eq:9} that there exists a hyperplane $H^*$ of $H_{\ta_r-1}$ with 
\begin{equation}\label{eq:10}
b_{H^*,1} \leq b_{avg,1} < \de_{t+1-\ta_r}.
\end{equation}
We then use Lemma~\ref{lem:HeLe0} and~\eqref{eq:dlt}
on the partition $\Pa_{\ta_r-1}$ and the hyperplane $H^*$ of $H_{\ta_r-1}$ to obtain:
\begin{align}\label{eq:11.1}
1+b_{H^*,1}\;q+\sum_{i=t-\ta_r+1}^t b_{H^*,i}\;q^i=|\Pa_{\ta_r-1}|
&=n_t+m_{\ta_r-1,1}\cr
&=\ell q^{t}+2+\de_{t+2-\ta_r}\\
&=1+\ell q^{t}+q\de_{t+1-\ta_r},\notag
\end{align}
Simplifying~\eqref{eq:11.1} yields
\begin{equation}\label{eq:11.2}
b_{H^*,1}+\sum_{i=t-\ta_r+1}^t b_{H^*,i}\;q^{i-1}=\ell q^{t-1}+\de_{t+1-\ta_r}.
\end{equation}
Then,~\eqref{eq:dlt} and~\eqref{eq:11.2} imply that
\begin{equation}\label{eq:12}
b_{H^*,1}\md{q^{t-\ta_r}}=\de_{t+1-\ta_r}.
\end{equation}
Since $t-\ta_r\geq 1$, it follows from~\eqref{eq:11.2} and~\eqref{eq:12} that 
$b_{H^*,1}\geq\de_{t+1-\ta_r}$, which 
contradicts~\eqref{eq:10}.
Thus, $\mu_q(n,t)\leq \ell q^t+1$ and the proof is complete.
\end{proof}
\bs

\begin{proof}[Proof of Theorem~\ref{thm1}]
For $r=0$, Theorem~\ref{thm1} is just the result of Andr\'e~\cite{An}, and for $r=1$, it follows from Theorem~\ref{BHP}.
For $r\geq 2$, Theorem~\ref{thm1} holds since the lower bound for 
$\mu_q(n,t)$ given in
Theorem~\ref{BHP} and the upper bound given in
Lemma~\ref{lem:main} are equal.
\end{proof}
%
\section{Concluding Remarks}\label{sec:conc}
Applying the same averaging method used in the proof of Lemma~\ref{lem:main}
substantially improves the upper bound given by Drake and Freeman 
(see Theorem~\ref{DF}) in some of the remaining cases, i.e.,
when $t\in[r+1,\ta_r]$. However, we omit those types of results here and will address them elsewhere\footnote{These results have now appeared in~\cite{NS3}.}. For instance, we can prove the following lemma. 
\begin{lemma}\label{lem:extra} Let $n$, $t$, and $r$ be  integers such that $0\leq r<t<n$ and $r=n\md{t}$. 
If $r\geq2$ and $t=\ta_r$, then $\mu_q(n,t)\leq \ell q^t+q$.
\end{lemma}
\begin{remark} If $n$, $t$, and $r$ satisfy the hypothesis of Lemma~\ref{lem:extra},
then (after some simplifications) Theorem~\ref{DF} yields 
$\mu_q(n,t)\leq \ell q^t+\left\lceil\frac{q^r}{2}\right\rceil$.
\end{remark}

 As mentioned in the introduction (Section~\ref{sec:1}), our result (Theorem~\ref{thm1}) settles 
almost all the remaining cases of one of the main 
unsolved problems related to partial $(t-1)$-spreads over $\pg(n-1,q)$. As a corollary, Theorem~\ref{thm1} also settles 
several open problems in the area of subspace coding 
that were raised by Etzion~\cite{Et2}, Etzion--Storme~\cite{EtSt}, and Heinlein et al.~\cite{HKKW}.
 
A {\em subspace code} over $\pg(n-1,q)$ is a collection of subspaces of $\pg(n-1,q)$
(e.g., see~\cite[Section~$4$]{EtSt} for a recent survey). In their seminal paper, K\"oetter and  Kschischang~\cite{KK} showed that subspace codes were well-suited for error-correction in the new model for 
information transfer called {\em network coding}~\cite{ACYLY}.
Partial $(t-1)$-spreads form an important class of subspace codes, called {\em Grassmannian codes} (e.g., see~\cite{KK,EV,GR}). 
Our result implies that the largest known partial $(t-1)$-spread codes are optimal for almost all values of $n$, $t$, and $q$. 
\begin{remark}
After submitting this paper, we learned from 
Ameera Chowdhury~\cite{C} that Theorem~\ref{thm1} also determines the clique number of the $q$-Kneser graph.
\end{remark}
The {\em Kneser graph}, $K(n,t)$, is the graph whose vertices are the 
$t$-element subsets of an $n$-set and with any two vertices adjacent if their corresponding subsets are disjoint. The graph $K(n,t)$ is well-studied in the context of extremal combinatorics. For instance, the chromatic number of $K(n,t)$ was determined by Lov\'asz~\cite{L},
and the maximum size of an independent set in $K(n,t)$ is given 
by the celebrated Erd\"os-Ko-Rado theorem~\cite{EKR}.  

The {\em $q$-analogue of the Kneser graph}, $K_q(n,t)$, is the graph whose vertices are the $t$-subspaces of $V(n,q)$ and with any two vertices adjacent if their corresponding $t$-subspaces have trivial intersection. Somewhat recently, the chromatic number of $K_q(n,t)$ has been essentially determined by Blokhuis et al.~\cite{B_al} and Chowdhury et al.~\cite{C_al}. On the other hand, the  maximum size of a independent set in $K_q(n,t)$ was given much earlier by Hsieh~\cite{H} and  Frankl-Wilson~\cite{FW}.  

Determining the clique number of the Kneser graph is trivial. However, the clique number of the $q$-Kneser graph was not known.  The main result of this paper (Theorem~\ref{thm1}) yields the clique number of $K_q(n,t)$ for $t$ large enough. 
\begin{cor}\label{cor1}
Let $r=n\md{t}$ and $0\leq r<t$. If $t>(q^r-1)/(q-1)$, then the clique number of $K_q(n,t)$ is 
\[\mu_n(n,t)=\frac{q^n-q^{t+r}}{q^t-1}+1.\]
\end{cor}

\bs\n{\bf Acknowledgement:}\
We thank Ameera Chowdhury for pointing out the connection of our work to $q$-Kneser graphs and for providing related references.


\end{document}